\documentclass{amsart}
\usepackage{enumerate}
\usepackage{colortbl}
\newtheorem{theorem}{Theorem}[section]
\newtheorem{lemma}[theorem]{Lemma}
\newtheorem{corollary}[theorem]{Corollary}
\newtheorem{proposition}[theorem]{Proposition}
\theoremstyle{definition}
\newtheorem{definition}[theorem]{Definition}
\newtheorem{example}[theorem]{Example}

\theoremstyle{remark}
\newtheorem{remark}[theorem]{Remark}
\theoremstyle{claim}
\newtheorem{claim}[theorem]{Claim}

\numberwithin{equation}{section}

\def\C{\mathbb C}
\def\R{\mathbb R}
\def\X{\mathbb X}
\def\Z{\mathbb Z}
\def\Y{\mathbb Y}
\def\Z{\mathbb Z}

\def\cal{\mathcal}

\begin{document}

\title[Existence of bounded asymptotic solutions]{
Existence of bounded asymptotic solutions of autonomous differential equations}

\author[V.T Luong]{Vu Trong Luong}
\address{VNU University of Education, Vietnam National University at Hanoi, 144 Xuan Thuy, Cau Giay, Hanoi, Vietnam}
\email{vutrongluong@gmail.com}

\author[W. Barker]{William Barker}
\address{Department of Mathematics and Statistics, University of Arkansas at Little Rock, 2801 S University Ave, Little Rock, AR 72204, USA}
\email{wkbarker@ualr.edu}

\author[N.D. Huy]{Nguyen Duc Huy}
\address{VNU University of Education, Vietnam National University at Hanoi, 144 Xuan Thuy, Cau Giay, Hanoi, Vietnam}
\email{huynd@vnu.edu.vn}

\author[N.V. Minh]{Nguyen Van Minh}
\address{Department of Mathematics and Statistics, University of Arkansas at Little Rock, 2801 S University Ave, Little Rock, AR 72204, USA}
\email{mvnguyen1@ualr.edu}

\thanks{This research is funded by VietNam National Foundation for Science and technology Development (NAFOSTED) under grant number 101.02-2023.17}

\date{\today}
\subjclass[2010]{Primary: 34G10; Secondary: 34C25, 34D05}
\keywords{Autonomous evolution equation on the half line, spectrum of a bounded function on the half line, analytic $C_0$-semigroup, sums of commuting operators, product of commuting semigroups, Katznelson-Tzafriri type theorem, Massera type result,  asymptotic solution}

\begin{abstract} 
We study the existence of bounded asymptotic mild solutions to evolution equations of the form $u'(t)=Au(t)+f(t), t\ge 0$ in a Banach space $\X$, where $A$ generates an (analytic) $C_0$-semigroup and $f$ is bounded. We find spectral conditions on $A$ and $f$ for the existence and uniqueness of asymptotic mild solutions  with the same "profile" as that of $f$. In the resonance case, a sufficient condition of Massera type theorem is found for the existence of bounded solutions with the same profile as $f$.
The obtained results are stated in terms of spectral properties of $A$ and $f$, and they are analogs of classical results of Katznelson-Tzafriri and Massera for the evolution equations on the half line.
Applications from PDE are given.
\end{abstract}

\maketitle

\section{Introduction, Notations and Preliminaries} \label{section 1}
In this paper we consider necessary and sufficient conditions for the existence and uniqueness (in some sense) of bounded solutions with specific structures of spectrum for equations of the form
\begin{equation}\label{eq}
	u'(t)=Au(t)+f(t),\ t\in J,
\end{equation}
where $J=[0,\infty)$, $A$ is the generator of an (analytic) $C_0$-semigroup in a Banach space $\X$, $f$ is a bounded continuous function on the half line $[0,\infty)$ with values in $\X$. As is well known in the qualitative theory of linear ODE (see e.g. \cite{cop}) that if $\X=\R^n$, and $A$ is a matrix with all eigenvalues lying off the imaginary axis, that means the equation has an exponential dichotomy, then there exists a projection $P$ in $\R^n$ such that $AP=PA$ and
\begin{align*}
	\| Pe^{tA}P \| \le Ne^{-\alpha t}, \ t\ge 0 ,\\
	\| Q e^{tA}Q\| \le Ne^{\alpha t}, \ t\le 0 ,
\end{align*}
where $Q:=I-P$, and $\alpha$ is a positive number independent of $t$. Therefore, in this case, the following function $u_0(\cdot)$ defined on $[0,\infty)$
\begin{equation}
	u_0(t)= \int^t_0 Pe^{(t-s)A}Pf(s)ds -\int^\infty_t Qe^{(t-s)A}Qf(s)ds 
\end{equation}
is a solution to Eq.(\ref{eq}) that is bounded on $J=[0,\infty)$. By Perron Theorem, see e.g. \cite[Proposition 2, p. 22]{cop}, the existence of at least a bounded solution on $\R^+$ for each given bounded and continuous $f$ yields the exponential dichotomy of (\ref{eq}), in this finite dimensional case, as well. In general, Eq.(\ref{eq}) may have more bounded solutions on the half line $[0,\infty)$. For example, all solutions of the form
\begin{equation}
	u(t)= Pe^{tA}P x+u_0, \ t\ge 0,
\end{equation}
are solutions of Eq.(\ref{eq}) that are bounded on $J=[0,\infty)$. It is easy to see that these are all solutions that are bounded on $[0,\infty)$. We note that actually, Eq(\ref{eq}) in this case has only one bounded solution within functions decaying to zero as $t\to +\infty$. We also note that as is well known (see \cite{cop}) for equations defined on the whole real line the situation is much more simple. In fact, $u_0(\cdot )$ is the unique solution that is bounded on the whole real line if $f(\cdot )$ is a given bounded on the whole real line. In the infinite dimensional case, additional conditions are needed for the equivalence between the exponential dichotomy and the existence of at least a bounded solutions on the half line. For more details, see e.g. \cite{massch}, \cite{minrabsch}.

\bigskip
It is also interesting to note that the bounded solution problem for Eq.(\ref{eq}) on the whole real line $J=\R$, where  $f$ is defined and bounded on the whole real line $\R$, turns out to be "simpler" as the necessary and sufficient conditions can be stated in a shorter manner. Namely, as is well known, in this case the Perron Theorem says that Eq. (\ref{eq}) has an exponential dichotomy if and only if for each $f$ defined and bounded on the whole real line $J=\R$ there exists a unique bounded solution on the real line $\R$. In the infinite dimensional case, the solutions may be understood in the sense of mild solutions.

\bigskip
Bounded Solution Problem of Eq.(\ref{eq}) when $J=\R$ has been considered in a broader context when the equation may have no exponential dichotomy, that is, the operator $A$ may allow $\sigma (A)\cap i\R \not= \emptyset$ in many works.  For more information see e.g. \cite{furnaimin,hinnaiminshi,min,murnaimin,naimin,naiminshi,naiminmiyshi,shinai}. Basically, the results can be summarized as follows: the concept of spectrum $\sigma (f)$ of a function $f$ on $\R$ is used that generalizes the concept of frequency. For autonomous equations (i.e. $A$ is independent of time $t$), if $A$ generates an analytic semigroup, then the non-resonance condition $\sigma (A)\cap i\R \not= \emptyset$ guarantees the existence of a classical solution $u_f$ with $\sigma (u_f) \subset \sigma (f)$. In the general case without analyticity of the semigroup $T(t)$ generated by $A$, a non-resonance condition is imposed on the operator $T(1)$ (see \cite{naimin}). In the resonance case a well known result (Massera Theorem) in the theory of ODE says that for Eq.(\ref{eq}) with $u(t)\in \R^n$ and $f(\cdot )$ being $p$-periodic it has a $p$-periodic solution if and only if it has a bounded solution on $\R^+$. In the general case, numerous extensions of this results are obtained for periodic solutions in  \cite{furnaimin,shinai, naiminmiyshi}, for almost periodic solutions in \cite{naiminshi}.

\bigskip
For Eq.(\ref{eq}) on the half line $\R^+$, the situation is much more complicated because we have very little information on the function $f$ on the half line. And the spectrum of a function on the half line does not capture much of the behavior of the function as in the whole real line case.
It is the purpose of this paper to extend some of the above mentioned results to the half line case when the eigennvalues of $A$ (or spectrum $\sigma (A)$, if $\X$ is generally a Banach space) may intersect the imaginary axis. This generalization will be based on the spectral theory of functions on the half line $[0,\infty)$ that has been developed recently in a series of papers by \cite{min,luomin,luoloiminmat} and a recent paper \cite{dinjanminngu}. We will use the method of sums of commuting operators as it is used in \cite{murnaimin} to prove the existence and uniqueness of asymptotic mild solutions. Our first result is stated in Theorems \ref{the main1} saying that in case $A$ generates an analytic $C_0$-semigroup, $f$ has precompact range and if $\sigma (A)\cap isp(f) =\emptyset$, where $sp (f)$ is the spectrum of $f$, then there exists an asymptotic mild solution to Eq.(\ref{eq}) that is unique in $BUC_C(\R^+,\X)$ within a function converging to zero. This result is a Katznelson-Tzafriri type theorem (see e.g. \cite{vu} for related results and topics).
We also consider the resonance case when $\sigma (A)\cap isp(f) \not=\emptyset$. In this case, the second main result of this paper of Massera type theorem (Theorem \ref{the main2}) says that under the assumption that $f\in BUC_C(\R^+,\X)$ and $ sp(f)\cap \overline{\sigma_i (A)\backslash sp(f)}=\emptyset$ and one of the sets $sp(f)$, $\overline{\sigma_i (A)\backslash sp(f)}$ is compact, where $\sigma_i(A)$ is the part of $\sigma (A)$ on the imaginery axis, then, then there exists an asymptotic mild solution with the same spectral profile as $f$ provided that there exists an asymptotic mild solution of Eq.(\ref{eq}) in $BUC_C(\R^+,\X)$. The obtained results complement those in \cite{luoloiminmat} and \cite{naiminshi}.

\section{ Notations and Preliminaries}
\subsection{Notations}
In this paper we will denote by $\R$, $\C$ the fields of real and complex numbers, respectively. By $\mathbb{X}$ we often denote a Banach space over $\mathbb{C}$ with norm $\|\cdot\|$. The space $\mathcal{L}(\mathbb{X})$ of all bounded linear operators on $\mathbb{X}$ with norm $\| \cdot \|$, by abuse of notation for our convenience if this does not cause any confusion. For a linear operator $A$ in $\mathbb{X}$ we denote by $D(A)$ its domain, and $\sigma (A)$ and $\rho(A)$ its spectrum and resolvent set, respectively. If $\mu\in \rho(A)$, then $R(\mu,A):=(\mu-A)^{-1}$. We will denote a $C_0$-semigroup of linear operators in a Banach space $\X$ usually as $(T(t))_{t\ge 0}$, but sometimes as $T(t)$ for simplicity, if this does not cause any confusion.

\subsection{A spectral theory of bounded functions on the half line $[0,\infty)$}
Let $\cal D$ denote the differentiation operator $d/dt$ in $BUC  (\R^+,\X)$ with domain
$$
D(\cal D) =\{ f\in BUC (\R^+,\X): \exists f', \ f'\in BUC  (\R^+,\X)\} .
$$
As is well known, the translation semigroup $(S(t)_{t\ge 0})$ in $BUC  (\R^+,\X)$, defined as $S(t)f (\cdot ):=f(t+\cdot )$ for each $f\in BUC(\R^+,\X)$, is strongly continuous in $BUC  (\R^+,\X)$ with ${\cal D}$ as its infinitesimal generator. 
\subsubsection{Operator $\tilde{\cal D}$}
Throughout the paper we will use the following notation
\begin{eqnarray*}
	C_{0}(\R^+,\X) &:=&\{ f\in BUC   (\R^+,\X): \ \lim_{t\to \infty} f(t)  =0\} .
\end{eqnarray*}
$C_{0  }(\R^+,\X)$ is a closed subspace $BUC   (\R^+,\X)$, and is invariant under the translation semigroup $(S(t)_{t\ge 0})$.
In the space $BUC   (\R^+,\X)$ we introduce the following relation $R$:
\begin{equation}
	f \ R \ g \ \mbox{if and only if} \ \ f -g \in C_{0  }(\R^+,\X) .
\end{equation}
This is an equivalence relation and the quotient space $\Y:= BUC   (\R^+,\X)/ R$ is a Banach space. We will also denote the norm in this quotient space $\Y$ by $\| \cdot \| $ if it does not cause any confusion.

\bigskip
Similarly we define the space $\Y^C:= BUC_C(\R^+,\X)/C_{0  }(\R^+,\X) ,$ where $$
BUC_C(\R^+,\X):= \{ f\in BUC_C(\R^+,\X): \ \mbox{the range of $f$ is precompact}\}.$$
 Note that $BUC_C(\R^+,\X)$ is a closed subspace of $BUC (\R^+,\X)$ that is left invariant by the semigroup of translations $(S(t)_{t\ge 0})$ in $BUC  (\R^+,\X)$.

\medskip
The class containing $f\in BUC   (\R^+,\X)$ will be denoted by $\tilde{f}$. Define $\tilde{\cal D}$ in $ \Y=BUC   (\R^+,\X)/ R$ as follows:
\begin{eqnarray}
	D( \tilde{\cal D}) &:=&\{ \tilde{f} \in \Y : \exists u\in \tilde{f}, u\in D(\cal D)\} .
\end{eqnarray}
If $f\in D( \tilde{\cal D}) $, we set
\begin{equation}
	\tilde{D}\tilde{f} := \widetilde {\cal Du} 
\end{equation}
for some $u\in \tilde f$. The following lemma will show that this $\tilde{D}$ is well defined as an operator in $\Y$. Similarly, we can define $\tilde{\cal D}$ on $\Y^C$.
\begin{lemma}
	The following assertions are valid:
	\begin{enumerate}
		\item With the above notations, $\tilde{\cal D}$ is a well defined single valued linear operator in $\Y$ and $\Y^C$;
		\item The induced semigroup $(\bar S(t)_{t\in\R^+})$ is extendable to a strongly continuous group $(\bar S(t)_{t\in\R})$ in $\Y$ with $\tilde{\cal D}$ as its infinitesimal generator.
	\end{enumerate}
	 \end{lemma}
\begin{proof}
	For the proof see \cite{arebat,luomin}.
\end{proof}

For each given $f\in BUC  (\R^+,\X)$ consider the following complex function $\hat f(\lambda )$ in $\lambda$ taking values in $\Y$ defined as
\begin{equation}
	\hat f (\lambda ) := (\lambda -\tilde{\cal D})^{-1} \tilde{f} .
\end{equation}

\begin{definition}\label{def spectrum}
	The set of all points $\xi_0 \in \R$ such that $\hat f (\lambda )$ has no analytic extension to any neighborhood of $i\xi_0$ is defined to be the spectrum of $f$, denoted by $sp  (f)$.
\end{definition}
\subsection{Arveson spectrum of a bounded function on the half line}
The concept of spectrum of a bounded function $g$ on the half line in Definition \ref{def spectrum} is actually the Arveson spectrum of $g$ with respect to the group of isometries $(\overline{S}(t)_{t\in\R})$ in $\Y$ (see e.g. \cite[p.365]{arebat2}) that is defined as follows: 
\begin{align}
	sp^{\bar S} (g)&:=\{\xi \in \R: \forall \varepsilon >0 \ \exists f\in L^1(\R) \ \mbox{such that} \nonumber \\
	&supp \overline{\cal F} f \subset (\xi-\varepsilon ,\xi +\varepsilon ) \ \mbox{and} \ f(\bar S_g)\not=0  \},
\end{align}
where $(\bar S_g(t)_{t\in\R})$ is the restriction of the group $(\overline{S}(t)_{t\in\R})$ to the invariant closed subspace that is the closure of the linear subspace ${\cal M}_g$ of $\Y$ spanned over all translations $\bar S(t)g, t\in \R$. We also recall the following definitions (see also\cite[p.365]{arebat2})
$$
f(\bar S_g)h:= \int^{+\infty}_{-\infty} f(t)\bar S(t)hdt , \ (h\in \cal M_g),
$$
and 
$$
\bar {\cal F}f(s):=\int^{+\infty}_{-\infty} e^{st}f(t)dt
$$
is the Fourier transform of $f$. Subsequently, the spectrum of $f\in BUC(\R^+,\X)$ can be determined by the spectrum of the operator $\tilde{\cal D}_f$ as follows (see e.g. \cite[p.366]{arebat2}):
\begin{lemma}\label{lem 2.3}
Let $f\in BUC(\R^+,\X)$. Then, the following are valid:
\begin{enumerate}
	\item 
	\begin{equation}
		sp(f) =\sigma (\tilde{\cal D}|_{\cal M_f}).
	\end{equation}
	\item 
	\begin{align}
		sp(f) =\{ & \xi \in \R: \forall \varepsilon >0 \ \exists \psi \in L^1(\R) \ \mbox{such that} \nonumber   \\
		            & supp \bar {\cal F} \psi \subset (\xi -\varepsilon ,\xi +\varepsilon), \psi * u=0 \} .
	\end{align}
	
\end{enumerate}

\end{lemma}

\begin{definition}
	Let $p$ be a given real number in $[0,2\pi)$. A function $g\in BUC(\R^+,\X)$ is said to be an asymptotic Bloch $1$-periodic function of type $p$ if
	\begin{equation*}
		\lim_{t\to\infty} (g(t+1)-e^{ip}g(t))=0.
	\end{equation*}
	If $p=0$, an asymptotic Bloch $1$-periodic function $g$ of type $p$ will be called an asymptotic $1$-periodic function. When $p=\pi$ we call the function asymptotic anti-periodic.
\end{definition}

In \cite{luoloiminmat} a function $g\in BUC(\R^+,\X)$ is an asymptotic Bloch $1$-periodic function of type $p$ if and only if $\sigma (g):=\overline{e^{i sp(f)}}\subset \{ e^{ip}\}$. In the following some characterizations of the asymptotic behavior of a function $g$ are given in terms of the spectral properties of $f$
\begin{theorem}\label{the 3.8}
	Let $g\in BUC  (\R^+,\X)$. Then,
	\begin{enumerate}
		\item
		If $\xi_0$ is an isolated point in $sp (g)$, then $i\xi_0$ is either removable or a simple pole of ${\hat g(\lambda)}$;
		\item If $sp  (g)=\emptyset$, then $g\in C_{0 }(\R^+,\X)$;
		\item $sp  (g)$ is a closed subset of $\R$;
		\item $g$ is an asymptotic Bloch $1$-periodic function of type $p$ if and only if $sp (g) \subset  p+2\pi \Z$.
	\end{enumerate}
\end{theorem}
\begin{proof}
	For the proofs of (i), (ii) and (iii) see \cite{luomin}. For (iv), note that as shown in \cite{luoloiminmat}, 
	$f$ is an asymptotic Bloch $1$-periodic function of type $p$ if and only if $\sigma (f) \subset \{ e^{ip}\}$. By the Weak Spectral Mapping Theorem, this means
	$$
	\overline{e^{isp(g)}} =\sigma (g) \subset \{e^{ip}\}.
	$$
This is equivalent to $sp (f) \subset p+2\pi \Z$.
\end{proof}

\subsection{Sums of commuting operators}
We recall the notion of two commuting operators which will be used in the sequel. 

\begin{definition}\label{def 2.2} 
	Let $A$ and $B$ be operators on a Banach space $G$ with non-empty resolvent set. We say that {\it A and B commute} if one of the following equivalent conditions hold: 
	\begin{enumerate} 
		\item \  $R(\lambda ,A)R(\mu ,B)=R(\mu ,B)R(\lambda ,A)$ for some (all) $\lambda \in \rho (A), \mu \in \rho (B) $ , 
		\item \  $ x \in D(A)$ implies $R(\mu ,B)x \in D(A)$ and $AR(\mu ,B)x= 
		R(\mu ,B)Ax$ for some (all) $\mu \in \rho (B)$. 
	\end{enumerate} 
\end{definition} 

For $ \theta \in (0,\pi ), R>0$ we denote $\Sigma (\theta ,R)=\{ z \in {\bf C}: | z| \ge R , |arg z |\le \theta \}$. 

\begin{definition}\label{def 2.3}
	\rm 
	Let $A$ and $B$ be commuting operators. Then 
	\begin{enumerate}
		\item \  $A$ is said to be of class  $ \Sigma (\theta +\pi /2 ,R) $ if  
		there are positive constants $\theta , R$ such that $ 0< \theta < \pi /2$, and  
		\begin{equation}\label{8} 
			\Sigma (\theta +\pi /2 ,R) \subset \rho (A)  \mbox{and}  
			\sup_{\lambda \in \Sigma (\theta +\pi /2 ,R)}\| \lambda R(\lambda ,A)\| < \infty , 
		\end{equation} 
		\item \  $A$ and $B$ are said to satisfy {\it condition P} if there are positive constants  
		$ \theta , \theta ', R, \theta ' <\theta $ such that  $A$ and $B$  are of class 
		$ \Sigma (\theta +\pi /2 ,R),\Sigma (\pi /2 -\theta ' ,R)$, respectively. 
		
	\end{enumerate} 
\end{definition} 
If  $A$  and $B$ are  commuting  operators, $A+B$ is  
defined by $(A+B)x=Ax+Bx$ with domain $D(A+B)=D(A)\cap D(B)$. 

In this paper we will use the following norm, defined by $A$ on the  
space $\X$, $\| x\|_{{\cal T}_{A}} := \| R(\lambda ,A)x\|  
$, where $\lambda \in \rho (A)$. It is seen that different  
$\lambda \in \rho (A)$ yields equivalent norms. We say that an operator 
$C$ on $\X$  is $A$-closed if its graph is closed with respect to the  
topology induced by ${\cal T}_A$ on the product ${\X}\times {\X}$. 
It is easily seen that $C$ is $A$-closable if  
$x_n\to 0, x_n\in D(C), Cx_n\to y $  
with respect to ${\cal T}_A$ in ${\X}$ implies $y=0$. In this case, 
$A$-closure of $C$ is denoted by $\overline{C}^A$.

\begin{theorem}\label{the 2.2} 
	Assume that  $A$ and $B$  commute. Then the following assertions hold: 
	\begin{enumerate} 
		\item \  If one of the operators is bounded, then 
		\begin{equation}\label{9} 
			\sigma (A+B) \subset \sigma (A) + \sigma (B). 
		\end{equation}
		\item \  If $A$ and $B$ satisfy condition P, then $A+B$ 
		is $A$-closable, and  
		\begin{equation}\label{10} 
			\sigma (\overline{(A+B)}^A)\subset \sigma (A) +\sigma (B). 
		\end{equation} 
		In particular, if $D(A)$ is dense in {\bf X}, then 
		$\overline{(A+B)}^A= \overline{A+B}$ , where $\overline{A+B}$  
		denotes the usual closure of $A+B$. 
	\end{enumerate} 
\end{theorem} 

\begin{proof} For the proof we refer the reader to \cite[Theorems 7.2, 7.3]{3}.
\end{proof} 
\subsection{Product of two $C_0$-semigroups of commuting operators}
Let $A$ and $B$ be the generators of two $C_0$-semigroups $S(t)$ and $T(t)$ in a Banach space $\X$ that commute with each other, that is, $S(t)T(t)=T(t)S(t)$ for all $t\ge 0$. Then, we have the following 
\begin{lemma}\label{lem comm}
 $R(t):= S(t)T(t)$ is a $C_0$-semigroup with generator $C=\overline{A+B}$	 
\end{lemma}
\begin{proof}
	For the proof see \cite[p.24]{nag}.
\end{proof}

\section{Main Results}
\subsection{Existence and uniqueness of bounded solutions}
\begin{lemma}\label{lem 3.1}
Let $\Lambda \subset \R$ be a closed subset. Then,
$$
\sigma (\tilde{\cal D}|_{\Y_{\Lambda}})=\Lambda 
$$
Similarly,
$$
\sigma (\tilde{\cal D}|_{\Y^C_{\Lambda}})=\Lambda 
$$

\end{lemma}
\begin{proof}
It suffices to show that given $\xi_0 \not\in \Lambda$, for each $f\in BUC(\R^+,\X)$ such that $\bar f\in\Y_{\Lambda}$ the equation
\begin{equation}\label{2.1}
	\bar u'(t)-\xi_0 \bar u(t)=\bar f(t), t\ge 0, 
\end{equation}
has a unique solution $\bar u\in \Y_{\Lambda}$. First, the uniqueness follows from \cite[Corollary 4.2]{luomin}. In fact, by \cite[Corollary 4.2]{luomin}, if there is two functions $\bar u,\bar v$, then,
$$
sp (\bar u-\bar v) \subset \{\xi_0\} ,
$$
while $sp (\bar u-\bar v)\subset \Lambda$, so $sp (\bar u-\bar v)=\emptyset$. That is, $\bar u-\bar v=0\in \Y_\Lambda$. 
The existence of such a function $\bar u$ follows from the fact that $sp(u)$ is actually the Arveson spectrum of $\bar u$ with respect to the isometric group $(\bar S(t)_{t\in\R})$ in $\Y$. That is, 
$$
isp(f) =\sigma (\tilde{\cal D}|_{\cal M_{f}})\subset i\Lambda .
$$
Therefore, $i\xi_0\in \rho (\tilde{\cal D}|_{\cal M_{f}})$, and this there exists a unique $\bar u \in {\cal M_{f}}$ as a solution to Eq.(\ref{2.1}). It is easy to see that as $\bar u\in {\cal M_{f}}$, $sp(\bar u)\subset sp(f)\subset \Lambda$. This shows that
$$
\sigma (\tilde{\cal D}|_{\Y_{\Lambda}})\subset i\Lambda .
$$
By choosing $u_\lambda (t)=e^{i\lambda t}x_0$ for each $\lambda \in \Lambda$ and $x_0\not=0$, we can show that $\overline{u_\lambda (\cdot)}$ is an eigenvector of 
$\tilde{\cal D}|_{\Y_{\Lambda}}$ with the eigenvalue $\lambda$, so $\lambda \in \sigma (\tilde{\cal D}|_{\Y_{\Lambda}})$. This completes the proof of the lemma.
\end{proof}

\begin{proposition}\label{pro 1}
Let $A$ be the generator of an analtic $C_0$-semigroup $(T(t)_{t\in \R^+})$. Then, the multiplication operators  $\bar T(t):BUC_C(\R^+,\X) \ni  f(\cdot ) \mapsto \overline{T}(t)  f(\cdot )
$ forms an analytic $C_0$-semigroup $(\bar {T}(t)_{t\ge 0})$ in $BUC_C(\R^+,\X)$, so the induced semigroup $\tilde T(t)$ is an analytic $C_0$-semigroup in $\Y^C$. Moreover, the generator $\bar A$ of this semigroup $(\bar {T}(t)_{t\ge 0})$ is the operator $\bar A_T$ defined as below:
\begin{align}
	D(\bar A_T)  := \{ f\in BUC_C(\R^+,\X)&: f(t) \in D(A), \forall t\in \R^+, Af(\cdot ) \in BUC_C(\R^+,\X)\} \nonumber \\
	(\bar A_T f)(t)&:= Af(t), \ \forall t\in \R^+ , \ f\in D(\bar A). \label{gen}
\end{align}
\end{proposition}
\begin{proof}
\begin{claim}
With the above notations $\bar T(t)$ is an analytic $C_0$-semigroup in $BCU_C(\R^+,\X)$.
\end{claim}
\begin{proof}
By \cite[Definition 5.1, p. 60]{paz} the analyticity of the multiplication (by $T(t)$)  semigroup $\bar T(t)$ follows from that of the semigroup $T(t)$ because
	$$
	[\bar T(t)f](\cdot ) := T(t)f(\cdot), \ f\in BUC_C(\R^+,\X), t\ge 0.
	$$
	Next, we show that $\bar T(t)$ is a $C_0$-semigroup. Set $K:=\overline{range (f)}$. Then, as $K$ is compact by \cite[Lemma 5.2]{engnag} the map
	\begin{equation}
		[0,1]\times K \ni (t,x) \mapsto T(t)x \in \X 
	\end{equation}
	is uniformly continuous. This yields that for each $\epsilon >0$ there exists a $\delta >0$ such that if $|t-t'|<\delta$ and  $\| y-y'\|<\delta$, where $t,t'\in[0,1], y,y'\in \X$, then
	\begin{equation}
	\| T(t)y-T(t')y'\| < \epsilon .
	\end{equation}
	Consequently,
	\begin{equation}
	\sup_{x\in \R^+}	\| T(t)f(x)-f(x)\| < \epsilon ,
	\end{equation}
	for all $0\le t <\delta$. In other words,
	\begin{equation}
		\lim_{t\to 0^+} \bar T(t) \bar f (\cdot ) =\bar f(\cdot ).
	\end{equation}
\end{proof}
Let $\bar A$ be the generator of the semigroup $(\bar T(t))_{t\ge 0}$. Let $f\in D({\bar A}_0)$. By evaluating the function $\bar A f(\cdot)$, for each $f\in D(\bar A)$, pointwise we see that
$$
{\bar A} f(x) = Af(x)=[\bar A_T f](x), \ x \in [0,\infty) .
$$
That is, $\bar A $ is part of the operator $\bar A_T$ defined in (\ref{gen}). That means, $D(\bar A) \subset D(\bar A_T)$ and $ \bar A f =\bar A_T f$ whenever $f\in D(\bar A)$.
\begin{claim}
With the above notations $\bar A=\bar A_T$.	
\end{claim}
\begin{proof}
First, we see that since $D(\bar A) \subset D(\bar A_T)$, it follows that $D(\bar A_T)$ is dense in $BUC_C(\R^+,\X)$. Next, as $T(t)$ is a $C_0$-semigroup there are positive constants $M,\omega$ such that let $\|T(t)\| \le Me^{\omega t}$ for all $t$. By the Hille-Yosida Theorem, this fact is equivalent to the following: for all real $\lambda >\omega$
\begin{align}
&\mbox{The resolvent set} \ \rho (A) \ \mbox{contains  the ray} \ (\omega, \infty)\label{3.7}\\ 
& \| R(\lambda ,A)\| \le \frac{M}{(\Re \lambda -\omega)^n} \ \mbox{for} \ \Re \lambda >\omega , \ n=1,2,\cdots .\label{3.8}
\end{align}
If $f(\cdot )\in BUC_C(\R^+,\X)$, then, $R(\lambda ,A)f(\cdot )\in BUC_C(\R^+,\X)$ as well. And we can verify easily that $\lambda \in \rho (\bar A_T)$ with $[R(\lambda,\bar A_T)f](\cdot ) =R(\lambda A)f(\cdot )$. That is why similar conditions to  (\ref{3.7}) and (\ref{3.8}) will be satisfied for $\bar A_T$, so $\bar A_T$ is the generator of a $C_0$-semigroup tat we may denote by $S(t)$. Next, we can see that this semigroup $S(t)$ must coincide with $\bar T(t)$. In fact, let $f(\cdot)\in D(\bar A)$, by the general theory of $C_0$-semigroup (see e.g. \cite[Theorem 1.3]{paz}), both $\bar T(t)f(\cdot)$ and $S(t)f(\cdot)$ are the unique solution of the Cauchy Problem $u'(t)=\bar A u(t)$, $u(0)=f(\cdot )$, so they must be the same, that is
$$
\bar T(t)f(\cdot ) =S(t)f(\cdot), \ f(\cdot )\in D(\bar A).
$$
As $\bar T(t)$ and $S(t)$ are bounded linear operators and $D(\bar A)$ is dense in $BUC_C(\R^+,\X)$, we have $\bar T(t)=S(t)$ for all $t\ge 0$. This yields that $\bar A=\bar A_T$.
\end{proof}
The proof of the proposition follows immediately from the above claims.
\end{proof}
\begin{definition}
Let $A$ and $f$ be as in Eq.(\ref{eq}). Then, a function $u\in BUC(\R^+,\X)$ is said to be a mild solution of Eq.(\ref{eq}) if the following identity holds true for all $t\ge 0$:
\begin{equation}\label{mi}
	u(t)=T(t)u(0) +\int^t_0 T(t-s)f(s)ds.
\end{equation}
 The function $w\in BUC(\R^+,\X)$ is said to be an asymptotic mild solution of Eq.(\ref{eq}) if there is a function $\varepsilon (\cdot )\in C_0(\R^+,\X)$ such that for all $t\ge 0$,
 \begin{equation}
 	u(t)=T(t)u(0) +\int^t_0 T(t-s)\left(f(s)+\varepsilon(s)\right)ds.
 \end{equation}
 The operator $\cal G$ is defined with its domain consisting of all $u\in BUC(\R^+,\X)$ such that there is $f\in BUC(\R^+,\X)$ so that $u$ is the solution in Eq.(\ref{mi})
\end{definition}
It is proved in \cite{luoloiminmat} that both $\cal G$ and $\bar {\cal G}$ are well defined in $BUC_C(\R^+,\X)$ and $\Y^C$. 
\begin{lemma}\label{lem 3.6}
The operator  ${\cal G}$ defined as above is a closed operator in $BUC_C(\R ^+,\X)$ and is an extension of the operator ${\cal D} -{\bar A}$.
\end{lemma}
\begin{proof}
Let $BUC_C(\R ^+,\X)\supset D(\cal G) \ni u_n(\cdot ) \to u(\cdot)\in BUC_C(\R ^+,\X)$  and $w_n(\cdot ):={\cal G}u_n(\cdot ) \to w(\cdot)\in BUC_C(\R ^+,\X)$ as $n \to \infty$.  We need to show that $u\in D({\cal G})$ and ${\cal G}u=w$. In fact, for each fixed $t\ge 0$, the strong continuity of the semigroup $(T(t)_{t\ge 0})$ yields that for each $t\ge 0$ the function
$[0,t]\ni s\mapsto T(t-s)w_n(s)$ is continuous in $[0,t]$ and bounded. Therefore, by the Lebesgue Dominated Convergence Theorem
$$
\lim_{n\to\infty} \int_0^tT(t-s)w_n(s)ds = \int^t_0 \lim_{n\to\infty} T(t-s)w_n(s) ds =\int^t_0 T(t-s)w(s)ds.
$$
Therefore, for each fixed $t\ge 0$,
\begin{align*}
u(t)= \lim_{n\to\infty} u_n(t) &=\lim_{n\to\infty} T(t)u_n(0) + \lim_{n\to\infty} \int^t_0 T(t-s)w_n(s)ds\\
&= T(t)u(0)+\int^t_0 T(t-s)w(s)ds.
\end{align*}
By definition $u\in D(\cal G)$ and $\cal Gu=w$. In other words, $\cal G$ is a closed linear operator. Next, a classical solution of the evolution equation $u'(t)=Au(t)$ must be a mild solution, so a simple interpretation of this fact yields that  $\cal D - \bar A \subset \cal G$.
\end{proof}
By the general theory of $C_0$-semigroups in quotient spaces (see e.g. \cite[p.61]{engnag}), as $\bar {\cal G}$ is closed, in $\Y^C$ we have
$$
\overline{\tilde{\cal D}-\tilde{\cal A} } \subset  \tilde {\cal G} .
$$

We are now ready to prove the following that is a main result of the paper
\begin{theorem}\label{the main1}
	Let $A$ be the generator of an analytic semigroup and $f\in BUC_C(\R^+,\X)$ such that
	$\sigma (A) \cap i \cdot sp(f) =\emptyset$. 
	Then, there exists an asymptotic mild solution $w(\cdot )$ of Eq.(\ref{eq}) with 
	$$
	sp(w) \subset sp(f)
	$$
	that is unique within a function $g\in C_0(\R^+,\X)$.
\end{theorem}
\begin{proof}	
	Consider the operators $\tilde{\cal D}$ and $\tilde{\cal A}$ in $\Y^C_{\Lambda}$, where $\Lambda :=sp (f)$. By the above arguments (Lemma \ref{lem 3.1} and the proof of Proposition \ref{pro 1}) we have $\sigma (\tilde{\cal A}) \subset \sigma (A) $ and $\sigma (\tilde{\cal D})\subset i\Lambda$. This yields that
	$$
	0\not\in  \sigma (\tilde{\cal D} ) -\sigma (\tilde{\cal A}).
	$$
	Therefore, as $\tilde A$ is the generator of an analytic $C_0$-semigroup, by Part (ii) of Theorem \ref{the 2.2},
	in $\Y^C_{\Lambda}$ we have 
	$$
0\not\in 	\sigma (\overline{\tilde{\cal D} -  \tilde{\cal A}}) \subset \sigma (\tilde{\cal D} ) -\sigma (\tilde{\cal A}).
	$$
That means, the operator $\overline{\tilde{\cal D} -  \tilde{\cal A}}$ is invertible. The next step is to interprete this fact to finish the proof. We see that for each $f\in \Y^C_\Lambda$ there exists a unique $\bar u \in \Y^C_\Lambda$ such that
$$
\overline{\tilde{\cal D}-\tilde{\cal A} } \bar u =\bar f .
$$
Therefore, since $\cal G$ is a closed extension of $\cal D - \bar A $
$$
\overline{\tilde{\cal D}-\tilde{\cal A} } \bar u = \bar {\cal G}\bar u = \bar f.
$$
This means there exists a element $u\in BUC_C(\R^+,\X)$ such that $\cal Gu\in \bar f$. This in turn yields that $u$ is an asymptotic mild solution of Eq.(\ref{eq}). The uniqueness of $u$ within a function in $C_0(\R^+,\X)$ is clear.
\end{proof}
\subsection{Massera type Theorem for bounded solutions}
Before proceed we recall the concept of evolution semigroup associated with a $C_0$-semigroup.
\begin{definition}
Let $T(t)$ be a $C_0$-semigroup. Then, the semigroup $(T^h)_{h\ge 0}$ in $BUC_C(\R^+,\X)$ defined as
\begin{equation}
	T^hf(\cdot ):= \begin{cases}
		T(h)f(t-h), \ \mbox{if} \ t\ge h,\\
		T(t)g(0), \ \mbox{if} \ 0\le t\le h .
	\end{cases}
\end{equation}  
\end{definition}
As shown in \cite[Lemma 3.8]{luoloiminmat}, we have the following:
\begin{lemma}\label{lem 3.9}
	$(T^h)_{h\ge 0}$ forms a $C_0$-semigroup
	in $BUC_C(\R^+,\X)$. Moreover, $-\bar {\cal G}$ in $\Y^C$ is the generator of the  semigroup $(\bar T^h)_{h\ge 0}$ induced by $(T^h)_{h\ge 0}$ in $\Y^C$.
\end{lemma} 
\begin{corollary}
	With the above notations we have
	\begin{equation}\label{3.12}
		\overline{-\tilde{\cal D}+\tilde{\cal A} } = - \tilde {\cal G} .
	\end{equation}
\end{corollary}
\begin{proof}
	The inclusion
	$$
	\overline{\tilde{\cal D}-\tilde{\cal A} } \subset  \tilde {\cal G} 
	$$
	is clear from Lemma \ref{lem 3.6}. By using the formula for the generator of the product of two commuting semigroups (see e.g.  \cite[p.24]{nag}) and Lemma \ref{lem 3.9} this inclusion turns out to be an equality. In fact, as is well known (see e.g. \cite{arebat2}) the induced translation semigroup $\tilde S(t)$ in $\Y^C$ is extendable to a group of isometries. It is easy to see that $\tilde T^h =\tilde S(-h)\tilde T(h)$ and $\tilde S(-h)\tilde T(h) = \tilde S(-h)\tilde T(h)$. Therefore, by Lemma \ref{lem comm}, (\ref{3.12}) is valid because $-\tilde {\cal G}$ is the generator of $\tilde T^h$, and $-\tilde {\cal D}$, $\tilde {\cal A}$ are the generators of $\tilde S(-h), \tilde T(h)$, respectively.
\end{proof}

Before we proceed, let us denote $\sigma_i(A):= \{ \xi \in \R: \ i\xi \in \sigma (A)\}$. 
\begin{lemma}\label{lem 3.12}
	Let $u$ be a bounded asymptotic mild solution of Eq.(\ref{eq}). Assume further that $u$ and $f$ have precompact ranges. Then,
	\begin{equation}
	sp(f) \subset 	sp(u) \subset \sigma_i(A)\cup sp(f).
	\end{equation}
\end{lemma}
\begin{proof}
The inclusion $sp(u) \subset \sigma_i(A)\cup sp(f)$ was proved in \cite{luomin}. For the inclusion $sp(f) \subset sp(u)$, we note that if $u$ is an asymptotic mild solution, then $\tilde u \in D(\tilde{\cal G})$, and $-\tilde{\cal G}u=f$. Further as 
$-\tilde{\cal G}$ is the generator of the evoltion semigroup $ \tilde T^h$, 
\begin{align*}
	\lim_{h\downarrow 0} \frac{\tilde T^h\tilde u-\tilde u}{h} 
	&= 	\lim_{h\downarrow 0} \frac{\tilde T^h\tilde u-\tilde u}{h}\\
	&=-\tilde{\cal G }u=-\tilde f .
\end{align*}
Since $\tilde T^h$ is the composition of a multiplication by $\tilde T(t)$ operator and a translation, we have $sp (T^hu)\subset sp(u)$. Next, for each $h>0$
$$
sp\left(\frac{\tilde T^h\tilde u-\tilde u}{h} \right) \subset sp(u).
$$
This yields 
$$
sp(f) = sp\left(\lim_{h\downarrow 0} \frac{\tilde T^h\tilde u-\tilde u}{h} \right) \subset sp(u).
$$
\end{proof}
Below we present necessary and sufficient conditions for the existence of asymptotic mild solutions with specific structures of spectrum. 
\begin{corollary}\label{cor 3.12}
Let $f$ be a function in $BUC_C(\R^+,\X)$ and $\{ p\} \not \subset e^{i sp(f) }$. Then, there exists no asymptotic 1-periodic function in $BUC_C(\R^+,\X)$ that is an asymptotic mild solution Eq.(\ref{eq}).
\end{corollary}
\begin{proof}
The proof is obvious in view of Lemma \ref{lem 3.12} and Theorem \ref{the 3.8}.
\end{proof}
\begin{theorem}\label{the main2}
Assume that $u$ is an asymptotic mild solution of (\ref{eq}). Moreover, assume that
	$sp(f) \cap \overline{ \sigma_i(A)\backslash sp(f)} =\emptyset$ and one of the sets  $\overline{ \sigma_i(A)\backslash  sp(f)} $, $ \sigma_i(A)=\emptyset$ is bounded (so is compact).	
	Then, there exists an asymptotic mild solution $w(\cdot )$ of Eq.(\ref{eq}) with 
	$$
	sp(w) \subset sp(f)
	$$
	provided that there is an asymptotic mild solution $u(\cdot )\in BUC_C(\R^+,\X)$ of Eq.(\ref{eq}).
\end{theorem}
\begin{proof} Without loss of generality we may assume that $\overline{ \sigma_i(A)\backslash  sp(f)} $ is compact.
By Lemmas \ref{lem 2.3} and \ref{lem 3.12}, and the spectral decomposition \cite[Proposition 1.16, p. 245]{engnag} if $g\in BUC_C(\R^+,\X)$ such that $sp(g) \subset \sigma_i(A)\cup sp(f)$  we have  decomposition
\begin{equation}\label{split}
	\cal M_g =\cal M_1 \oplus \cal M_2 ,
\end{equation}
where 
$\cal M_1 $ is the spectral space corresponding to the spectral set $\overline{ \sigma_i(A)\backslash  sp(f)} $ (that is compact)
and the Riesz projection
$$
P_c :=\frac{1}{2\pi i}\int_{\gamma }R(\lambda , \tilde{\cal D} )d \lambda  
$$
is a bounded linear operator that commutes with $R(\lambda,A)$ and
\begin{align}
	\sigma (\tilde{\cal D}_1)&\subset  \overline{\sigma_i(A)\backslash sp(f)},\\
	\sigma (\tilde{\cal D}_2)&\subset sp(f).	
\end{align}
Let us define $\tilde u_1:= P_c u$ and $\tilde u_2:= (I-P_c)u$. Then, it is easily seen that $sp( \tilde u_1)\subset sp(f)$ and $sp(\tilde u_2) \subset \overline{ \sigma_i(A)\backslash sp(f)}$. It is also seen that as $R(\lambda ,\tilde {\cal D})$ commutes  with the semigroup $\tilde T(t)$, $P_c$ commutes with  the semigroup $\tilde T(t)$ as well. This means,  the semigroup $\tilde T(t)$ leaves the decomposition (\ref{split}) invariant. Without loss of the generality we may assume that $sp(f)$ is compact. Next, as $u$ is an asymptotic mild solution of Eq.(\ref{eq}), $\tilde u \in D(\tilde{\cal G})$, so
\begin{align*}
	\lim_{h\downarrow 0} \frac{T^hP_c\tilde u-P_c\tilde u}{h} 
  &= 	P_c\lim_{h\downarrow 0} \frac{T^h\tilde u-\tilde u}{h}\\
  &=-P_c \tilde f.
\end{align*}
This shows that $\tilde u_1=P_c\tilde u \in D(-\tilde{\cal G})$ and $ -\tilde{\cal G}P_c\tilde u =\tilde f$ . That means, $P_cu$ is an asymptotic mild solution of Eq.(\ref{eq}). The case $\overline{\sigma_i(A)\backslash sp(f)}$ is compact can be treated in the same lines.
\end{proof}

\bigskip\noindent
\begin{remark}
\begin{enumerate}
	\item In view of the failure of the Spectral Mapping Theorem for general $C_0$-semigroups
	the condition in the assertion (i) is a little more general than that formulated in
	terms of $\sigma (T(1))$. 
	\item If we know beforehand that $u$ is almost periodic, then in the statement of Theorem 3.5 we
	can claim that the spectral component $w$ is almost periodic.
\end{enumerate}
\end{remark}

\section{Examples}
\begin{example}
	The function $g(t)=\sin (\sqrt{t})$ is an asymptotic $1$-periodic function, but it cannot be expressed as the sum of an almost periodic function and a decaying function, that is, there are no two functions $a(\cdot )\in AP(\R,\X)$ and $c(\cdot )\in C_0(\R^+,\X)$ such that
	$$
	g(t)=a(t)+c(t), \ t\ge 0.
	$$
\end{example}
In \cite{luoloiminmat} it is proved that $g$ cannot be expressed as the sum of a $1$-periodic function and a decaying function. We can show easily that this function $g$ is asymptotic $\omega$-periodic for any $\omega >0$, so by the spectral theory in the previous section $sp(g) \subset 2\omega \pi \Z$ for any $\omega>0$. This yields that $sp(g) \subset \{0	\}.$
\begin{example}
	Consider the equation equation
	\begin{equation}
		u'(t) =Au(t)+f(t), t\ge 0,
	\end{equation}
	in a Banach space $\X$ where $A$ is the generator of analytic $C_0$-semigroup and $f$ is a function with precompact range. Then, as is well known (see e.g. \cite[Corollary 3.7]{paz}) $\sigma_i(A)$ should be bounded. Therefore, $\overline{\sigma_i\backslash sp(f)}$ must be compact. To apply Theorem \ref{the main2} we just need to require that $\overline{\sigma_i\backslash sp(f)} \cap sp(f) =\emptyset .$
\end{example}

\begin{example}\label{exa 2}
\end{example}
Consider the equation
\begin{equation}
	\begin{cases}
		w_t(x,t)=w_{xx}(x,t)+aw(x,t) + g(x)\sin(\sqrt{t}), \\
		\hspace{5cm}0\le s \le \pi ,\ t\ge 0, \cr
		w(0,t)=w(\pi ,t)=0 , \ \forall t > 0,
	\end{cases}
\end{equation}
where $1\ge a>0$ is given and $w(x,t), g(x)$ are scalar-valued functions, $g\in L^2[0,\pi ]$. We define the space 
${\X}:=L^2[0,\pi ]$ and $A_T: {\X}\to {\X}$ by the formula
\begin{equation}\label{exa 3}
	\begin{cases}\hspace{.5cm}
		A=y''-ay, \cr
		D(A)=\{ y\in {\X}:\ \mbox{y, y' are absolutely continuous}, \ y''\in {\X},\cr
		\hspace{4cm} y(0)=y(\pi )=0\}.
	\end{cases}
\end{equation}
The evolution equation we are concerned with
in this case is the following
\begin{equation}\label{ex}
	{\frac{dx(t)}{dt}}=Ax(t)+ f(t), \ x(t) \in {\bf X},
\end{equation}
where $A$ is the infinitesimal generator of an analytic and compact semigroup $(T(t))_{t\ge 0}$ in
${\X}$ (see \cite[p. 414]{traweb}) and $f(t):=\sin(\sqrt{t}) g(\cdot)$. Moreover, the eigenvalues of $A$ are
$-n^2, n=1,2,...$ and the set $\sigma_i(A )$ is determined from the set of
imaginary solutions of the equations
\begin{equation}\label{exa char}
	\lambda -a=-n^2, \ n=1,2,...\ .
\end{equation}
This yields that $\sigma_i(A)=\{ \sqrt{n^2-a}, n=1,2,\cdots \}$. 
\begin{enumerate}
	\item If $1>a>0$, then $sp(f) \cap \sigma_i(A)=\emptyset$. Theorem \ref{the main1} says that Eq.(\ref{ex}) has an asymptotic mild solution $u$ with $sp(u)\subset \{0\}$.
	\item If $a=1$, then, $sp(f)\cap \sigma_i(A)=\{0\}$. Theorem \ref{the main2} says that if Eq.(\ref{ex}) has an asymptotic mild solution $u\in BUC_C(\R,\X)$, then it has an asymptotic mild solution $w$ with
	$sp(w)\subset \{0\}$.
\end{enumerate}

\begin{example}
\end{example}
Consider the equation
\begin{equation}\label{exa 3}
	\begin{cases}
		w_t(x,t)=w_{xx}(x,t)+aw(x,t) + g(x) \cos(t), \\
		\hspace{5cm}0\le s \le \pi ,\ t\ge 0, \cr
		w(0,t)=w(\pi ,t)=0 , \ \forall t > 0,
	\end{cases}
\end{equation}
where $a$, $g$ satisfy the same conditions as in Example \ref{exa 2}. Further, we assume that $g(\cdot)\not= 0 \in L^2([0,\pi])$. In this case setting $f(t):= g(\cdot )\cos(t)$, for $t\ge 0$, we can show easily that $sp(f)=\{1\}$. By Corollary \ref{cor 3.12} we can claim that Eq.(\ref{ex}) does not have mild solution $u\in BUC_C(\R^+,\X)$ with $\lim_{t\to\infty} (u(t+1)-u(t))=0$. In fact, such a solution $u$, if exists, must satisfy $\{1\} \subset sp(u) \subset \{ 2\pi\Z\}$, so $1\in 2\pi\Z$, a contradiction.

%%%%%%%%%%%%%%%%%%%%%%% REFERENCES %%%%%%%%%%%%%%%%%%%%%%%%%%
\bibliographystyle{amsplain}

\end{document}